\documentclass[letterpaper, 10 pt, journal, twoside]{IEEEtran}

\usepackage[noadjust]{cite}

\usepackage[english]{babel}
\usepackage[final,pdftex]{graphicx}

\usepackage[cmex10]{amsmath}
\usepackage{amsfonts}
\usepackage{amssymb}
\usepackage{amsthm}
\usepackage{verbatim}
\usepackage{mathbbol}	
\usepackage{algorithm}
\usepackage{algorithmic}
\usepackage{booktabs}
\usepackage{color}
\usepackage{soul}

\usepackage[T1]{fontenc}

\newtheorem{theorem}{Theorem}
\newtheorem*{corollary}{Corollary}

\newtheorem*{remark}{Remark}
\newtheorem{definition}{Definition}

\newtheorem{lemma}{Lemma}

\usepackage{enumerate}
\usepackage{array}

\usepackage{url}

\newlength{\IEEECOLWIDTH}
\usepackage{booktabs}

\def\TITLE{A Robustness Measure of Transient Stability under Operational Constraints in Power Systems}
\def\TITLEnl{A Robustness Measure of Transient Stability under Operational Constraints in Power Systems}
\def\AUTHORS{Liviu Aolaritei, Dongchan Lee, Thanh Long Vu and Konstantin Turitsyn}

\usepackage[hyperindex=true, %
  pdftitle={\TITLE},%
  pdfauthor={\AUTHORS},%
  colorlinks=true,%
  pagebackref=false,%
  plainpages=false,%
  pdfpagelabels,%
  linkcolor=black,%
  citecolor=black,%
  filecolor=black,%
  urlcolor=black%
]{hyperref} 

\def\1{\mathbf{1}}
\def\0{\mathbf{0}}

\DeclareMathOperator{\diag}{diag}

\begin{document}

\title{\TITLEnl}

\author{\AUTHORS
}

\author{\AUTHORS
\thanks{Liviu Aolaritei is with the Automatic Control Laboratory,
	ETH Zurich, 8092 Zurich, Switzerland. Email: \texttt{\small	aliviu@ethz.ch.}}%
\thanks{Dongchan Lee, Thanh Long Vu and Konstantin Turitsyn are with the Mechanical Engineering Department, Massachusetts Institute of Technology, Cambridge, MA. Email: \texttt{\small	\{dclee, longvu, turitsyn\}@mit.edu.}}%
}

\maketitle

\thispagestyle{empty}

\begin{abstract}
The aggressive integration of distributed renewable sources is changing the dynamics of the electric power grid in an unexpected manner. As a result, maintaining conventional performance specifications, such as transient stability, may not be sufficient to ensure its reliable operation in stressed conditions. In this paper, we introduce a novel criteria in transient stability with consideration of operational constraints over frequency deviation and angular separation. In addition, we provide a robustness measure of the region of attraction, which can quantify the ability of the post-fault system to remain synchronized even under disturbances. To assess this new stability specification, we adopt the notion of Input-to-State Stability (ISS) to the context of power systems and introduce a new class of convex Lyapunov functions, which will result in tractable convex-optimization-based stability certificates. As a result, we are able to quantify the level of disturbance a power system can withstand while maintaining its safe operation. We illustrate the introduced stability specification and certificate on the IEEE 9 bus system.   
\end{abstract}

\begin{IEEEkeywords}
Power systems, Stability of nonlinear systems, Lyapunov methods, Uncertain systems, Robust control.
\end{IEEEkeywords}

\section{Introduction}
\label{sec:introduction}

\IEEEPARstart{T}{he} electric power grid is undergoing the most substantial transformation since its emergence. The large scale integration of distributed renewable sources introduces significant uncertainty into the grid's operation and reduces the aggregate inertia, hence, reducing the ability of the grid to counteract disturbances. At the same time, the risk of contingencies is growing with the increasing number of extreme weather events, threatening the stability and security of the electric power grid. Therefore, new assessment and control tools are needed to detect and mitigate warning behaviours expected in the dynamics of the ongoing and future power grids. 

In traditional power grids, transient stability assessment determines whether a power system maintains its synchrony after experiencing a large disturbance, i.e., if the frequencies of the machines and loads in the post-fault dynamics transiently converge to the nominal value of 50 or 60 Hz and their angles transiently converge to some stable steady state values \cite{Chiang:2011eo,Chiang:1994cUEP,hiskens1989energy}. In the modern power grid, this conventional transient stability certificate may not be sufficient to ensure a reliable operation of the system. Indeed, the high penetration of renewable generation can significantly unbalance the power supply and demand, leading to large frequency deviation, e.g., more than 0.5 Hz on a 50 or 60 Hz network. Large frequency deviation in turn will activate frequency protective relays on the power system network to shed loads/generators or trip interconnection lines \cite{Ingleson2010}. Since the widespread protective relays are still poorly coordinated due to the large scale of the grid, load/generator shedding and line tripping, in the worst case, may result in further cascading failures and even power blackout. 

To avoid the aforementioned harmful responses, we propose to consider a \emph{safe and robust transient stability} specification, which will not only ensure the conventional transient stability, but also guarantee that the frequencies and angles are within the operational constraints, i.e., the frequencies are in a small neighborhood of the nominal value, and the angular differences do not greatly exceed $\pi/2$. In addition, the system has to be robustly stable with respect to small disturbances in power injections, which are always present in the modern power system with high penetration of renewables.

A similar problem was considered in \cite{lee18}, where the same authors proposed an input-output stability framework to quantify the robustness against disturbances of the system operating at an equilibrium point. Differently, in this paper, the robustness against disturbances of the post-fault trajectory is considered, where the fault-dependent scenario is included in the assessment.

Technically, to analyze these new power system stability specifications, this paper brings the following novelties:

\begin{itemize}
    \item First, we introduce a new class of Lyapunov functions for the power system transient stability analysis. Unlike the energy function and usual Lyapunov functions for the power system stability analysis \cite{Chiang:2010gi,VuTuritsyn:2014}, the Lyapunov functions proposed here are convex on the whole state space. This key property leads to a \emph{convex-optimization-based stability certificate}, which allows us to quickly assess the safe and robust transient stability of the power system. 
    
    \item Second, we develop the ideas of Lyapunov function family approach \cite{VuTuritsyn:2014,VuTAC2017} to our new class of convex Lyapunov functions characterized by the solution of a linear matrix inequality (LMI). Each function in this new family of convex Lyapunov functions can provide an estimate of the stability region within the operational constraints, offering the adaptation of our stability certificate to given contingencies \cite{VuTuritsyn:2014} and reducing the conservativeness of the other Lyapunov function-based stability analyses.
    
    \item Third, we use Input-to-State Stability (ISS) arguments \cite{Sontag, dashkovskiy2007iss} to analyze the \emph{stability and robustness of the post-fault dynamics} with respect to disturbances in the power injections. In particular, we establish a measure of robustness for transient stability, which quantifies the maximum disturbance that the post-fault dynamics can incorporate, while maintaining its transient stability within the operational constraints. Due to the physics of the nonlinear power flow, there are possibly several equilibrium points that co-exist with their own stability region, therefore, the electric power system lacks global stability. As a result, in our analysis, we consider a local version of ISS \cite{mironchenko2016local}.
\end{itemize}

The remaining of this paper is organized as follows. In Section II, we present the model of the multi-machine power system, and we mathematically formulate the considered problem. In Section III, we introduce the new class of convex Lyapunov functions, which we then use to assess the transient stability under operational constraints. Section IV presents a local version of ISS, which is then used to compute a bound on the maximum disturbance that can enter in the post-fault dynamics. Finally, Section V demonstrates our stability certificates on an IEEE 9 bus system, and Section VI concludes the paper.

\section{Power system model}
\label{sec:model}

In the following we will make use of the following notation. We denote by $\mathbb{0}$, $\mathbb{1}$, and $I$ the zero matrix, the all-ones vector, and the identity matrix, of appropriate dimensions, respectively. Given a vector $x$ and two matrices $A$ and $B$, let $\diag(x)$ and $\diag(A,B)$ denote the diagonal matrix with the elements of $x$ on the diagonal, and the block-diagonal matrix with the matrices $A$ and $B$ on the diagonal, respectively. Moreover, we will use the notation $\sigma_{\text{min}}(A)$ and $\sigma_{\text{max}}(A)$ to denote the minimum and the maximum eigenvalue of $A$, respectively. The inequalities $\prec$, $\preceq$, $\succ$ and $\succeq$ define the matrix inequalities. Finally, we denote by $\lVert \cdot \rVert$, $\lVert \cdot \rVert_{\infty}$, and $\lVert \cdot \rVert_{\mathcal{L}_{\infty}}$ the Euclidean, infinity, and $\mathcal{L}_{\infty}$ norms, respectively.

\subsection{Swing equation model}
\label{subsec:swing}

The aging power grid with stressed load suffers from several contingencies. During a contingency, the system evolves according to the fault-on dynamics, moving away from the pre-fault equilibrium point. After the contingency is cleared or self-clears, the system experiences the post-fault dynamics. The post-fault dynamics is said to be transiently stable if the state of the system converges to a stable post-fault equilibrium point. In this paper, we consider the swing equation model to describe the post-fault dynamics of the multi-machine power system. This is a simplified dynamic model, that focuses on the relationship between the active power and the angles over the lossless power network with constant voltages.

A generator $i$ is characterized by its rotor angle $\theta_i$ and its angular velocity $\omega_i=\dot{\theta_i}$, and its dynamics is described by the following equations:
\begin{align}
\label{eq:swingmodel}
\begin{split}
&\dot{\theta_i} = \omega_i \\
&m_i \dot{\omega_i} + d_i \omega_i = P_{i}-P_{e,i}
\end{split}
\end{align}
where $m_i$ and $d_i$ are its dimensionless moment of inertia and damping action, respectively. Moreover, $P_{i}$ and $P_{e,i}$ are its effective dimensionless mechanical torque acting on the rotor and its effective dimensionless electrical power output, respectively. The electrical power output is given by
\begin{equation}
\label{eq:electricalpower}
P_{e,i} = G_i V_i^2 + \sum\limits_{j \in \mathcal{N}_i} y_{ij} \sin \theta_{ij} 
\end{equation}

Here, $\theta_{ij}=\theta_i-\theta_j$, and $y_{ij} = B_{ij} V_i V_j$, where $B_{ij}$ is the (normalized) susceptance of the line connecting the generators $i$ and $j$. The value of $V_i$ represents the voltage magnitude at the terminal of the $i^{\text{th}}$ generator and it is assumed constant. Finally, $\mathcal{N}_j$ is the set of neighboring generators of the $i^{\text{th}}$ generator.

The multi-machine power system is described by an undirected graph $\mathcal{G}(\mathcal{N}, \mathcal{E})$, where $\mathcal{N} = \{1,\cdots, n\}$ is the set of generators and $\mathcal{E}\subseteq \mathcal{N} \times \mathcal{N}$ is the set of transmission lines connecting the $n$ generators. Let $E$ denote the incidence matrix of the graph.

We define $\theta$, $\omega$, $P$ as the vectors obtained by stacking the scalars $\theta_i$, $\omega_i$, $P_i$, respectively, for $i \in \mathcal{N}$. Moreover, we define the diagonal matrices $M$ and $D$, which have the elements $m_i$ and $d_i$ on the diagonal, respectively, for $i \in \mathcal{N}$. Finally, we define the diagonal matrix $Y$, which has the elements $y_{ij}$ on the diagonal, for $(i,j) \in \mathcal{E}$.

Using the vector notation, the multi-machine power system can be described by:
\begin{align}
\begin{split}
&\dot{\theta} = \omega \\
&M \dot{\omega} + D \omega = P- E Y \sin\left(E^T \theta\right)
\end{split}
\label{eq:model}
\end{align}

The system \eqref{eq:model} is invariant under the transformations $\theta \rightarrow \theta + \tilde{\omega} t$, $\omega \rightarrow \omega + \tilde{\omega}$ and $P \rightarrow P + D \tilde{\omega}$ for any constant vector $\tilde{\omega} \in \mathbb{R}^n$, $\tilde{\omega} \in \text{Span}\{\mathbb{1}\}$. Without loss of generality, we assume that the angular velocity $\omega$ is defined with respect to the synchronous rotations $\tilde{\omega}$.
In normal conditions, the system reaches a synchronized state where all the machines rotate with the same angular velocity, and the angles of the machines are following a circular trajectory with $\theta = \theta^* + \tilde{\omega} t$ with $\theta^*$ satisfying the following algebraic equations:
\begin{equation}
\label{eq:equil}
P = E Y \sin\left(E^T \theta^*\right)
\end{equation}

The equilibrium point $\left[\theta^*,\mathbb{0}\right]$ is not unique, since every uniform shift $c \in \mathbb{R}$ in the rotor angles, $\left[\theta^* + c \mathbb{1},\mathbb{0}\right]$, results in another equilibrium point. However, the equilibrium point is uniquely characterized by the angle differences contained in the vector $E^T \theta^*$, and which solve the system \eqref{eq:equil}. In order to have an unique equilibrium point, we assume the existence of an infinite bus in the system, which will be used as a reference angle.

In practical situations the steady state equilibrium corresponds to relatively small deviations in angles, so that generally $\lVert E^T \theta^*\rVert_{\infty} < \pi / 2$ \cite{Dorfler:2013}.

\subsection{Lur'e system representation}
\label{subsec:lure}

The system \eqref{eq:model} can be rewritten in a Lur'e form, i.e., as a linear system with a nonlinear state feedback. To do so, we linearize the system around the equilibrium $\left[\theta^*,\mathbb{0}\right]$, and, by defining the system state as $x = [\left(\theta-\theta^*\right)^T, \omega^T]^T$, the network equations \eqref{eq:model} can be rewritten in a compact form as $\dot{x}=A x - B \phi(z)$, $z = C x$, where $\phi$ represents the nonlinear state feedback, and the matrices $A$, $B$, $C$ are equal to:
\begin{equation*}
A=
\begin{bmatrix}
\mathbb{0} & I \\
-M^{-1} E Y \diag\left(\cos\left(E^T \theta^*\right)\right)E^T & -M^{-1} D
\end{bmatrix},
\end{equation*}
\begin{equation*}
B=
\begin{bmatrix}
\mathbb{0} \\
M^{-1} E Y
\end{bmatrix},
\quad 
C=
\begin{bmatrix}
E^T & \mathbb{0}
\end{bmatrix}
\end{equation*}
where the nonlinearity $\phi \in \mathbb{R}^{|\mathcal{E}|}$ is composed by the following elements $\phi_{k} = \left(\sin \theta_{ij}-\sin\theta_{ij}^*\right) - \cos\theta_{ij}^*\left(\theta_{ij} - \theta_{ij}^*\right)$, where $k \in \{1,\ldots,|\mathcal{E}|\}$ corresponds to the edge $(i,j)$ in the network. Notice that the nonlinearity $\phi$ is decentralized, i.e., $\phi(z) = [\phi_1(z_1),\ldots,\phi_{|\mathcal{E}|}(z_{|\mathcal{E}|})]$.

The analysis carried out in this paper is based on the observation that the nonlinearity $\phi_k(z_k)$ can be lower and upper bounded by two linear functions $\underline{\delta}_k z_k$ and $\overline{\delta}_k z_k$, i.e.,
\begin{equation}
\label{eq:sector}
\frac{\phi_k(z_k)}{z_k} \in \left[\underline{\delta}_k,\overline{\delta}_k\right]
\end{equation}
where $\underline{\delta}_k$ and $\overline{\delta}_k$ are functions of the set where $z_k$ is restricted. In other words, by restricting the values of $z_k$ to smaller sets, tighter bounds on the nonlinearity $\phi_k(z_k)$ can be obtained.

In our case, $z_k = \theta_{ij}-\theta_{ij}^*$, so restricting $z_k$ translates into restricting the differences between the angles of the neighboring generators. From a practical perspective, large angle differences are strongly undesired, and it is common practice to consider angle differences which do not greatly exceed $\pi/2$. In the following lemma, we show how the nonlinearity can be bounded as in \eqref{eq:sector} in two sets of practical interest.

\begin{lemma}
\label{lemma:sector}
The nonlinearity $\phi_k(z_k)$ is bounded by the linear functions $\underline{\delta}_k z_k$ and $\overline{\delta}_k z_k$, where 
\begin{enumerate}
	\item[i)] $\underline{\delta}_k = -\cos\theta_{ij}^*$ and $\overline{\delta}_k = 1-\cos\theta_{ij}^*$ inside the set $\mathcal{P}_1 = \{x: \; \theta_{ij} \in [-\pi- \theta_{ij}^*,\pi- \theta_{ij}^*], \; \forall \; (i,j) \in \mathcal{E}\}$.
	\item[ii)] $\underline{\delta}_k = \xi_k-\cos\theta_{ij}^*$ and $\overline{\delta}_k = 1-\cos\theta_{ij}^*$ inside the set $\mathcal{P}_2 = \; \{x: \theta_{ij},\theta_{ij}^* \in [-\pi/2,\pi/2], \; \forall \; (i,j) \in \mathcal{E}\}$, where
\end{enumerate}
\begin{equation}
\label{eq:xi}
\xi_k = \frac{1-\sin | \theta_{ij}^*|}{\pi/2 - |\theta_{ij}^*|}
\end{equation}
\end{lemma}
\begin{proof}
Since $\phi_k(z_k)/z_k$ is equal to
\begin{equation*}
\frac{\sin\theta_{ij}-\sin\theta_{ij}^*}{\theta_{ij} - \theta_{ij}^*} - \cos\theta_{ij}^*,
\end{equation*}
and $|\sin\theta_{ij}-\sin\theta_{ij}^*|\leq |\theta_{ij} - \theta_{ij}^*|$, $\forall \; \theta_{ij}$, the upper bound can be chosen as $\overline{\delta}_k = 1-\cos\theta_{ij}^*$ inside both $\mathcal{P}_1$ and $\mathcal{P}_2$.

Inside $\mathcal{P}_1$, the function $\sin\theta_{ij}-\sin\theta_{ij}^* = 0$ in the equilibrium $\theta_{ij} = \theta_{ij}^*$, and on the boundary. Therefore, the lower bound can be chosen as $\underline{\delta}_k = -\cos\theta_{ij}^*$. 

Moreover, $\sin\theta_{ij}-\sin\theta_{ij}^*$ has a maximum in $\theta_{ij} = \pi/2$, equal to $1-\sin\theta_{ij}^*$, and a minimum in $\theta_{ij} = -\pi/2$, equal to $-1-\sin\theta_{ij}^*$. Therefore, inside $\mathcal{P}_2$, the lower bound can be chosen as $\underline{\delta}_k = \xi_k-\cos\theta_{ij}^*$, with $\xi_k$ defined in \eqref{eq:xi}.
\end{proof}



The lower bounds in Lemma \ref{lemma:sector} can be tighten if operational constraints on $\theta_{ij}$, $\forall (i,j) \in \mathcal{E}$, are imposed. Indeed, when restricting the angle differences inside a set $\mathcal{P}$, with $\mathcal{P}_2 \subset \mathcal{P} \subset \mathcal{P}_1$, an optimized lower bound $\underline{\delta}_k = \varphi_k-\cos\theta_{ij}^*$ can be found for the nonlinearity $\phi_k$, with $\varphi_k \in (0,\xi_k)$.

\begin{remark}
The stability analysis proposed in this paper holds for a general Lur'e system. This allows the direct extension of our methodology and results to more general power system models, which can be written in a Lur'e form with the nonlinearity composed by the same elements. For example, in \cite{lee18} we consider a structure-preserving model with first order turbine governor dynamics.
\end{remark}

\subsection{Problem formulation}
\label{probform}

 We consider practical operational constraints defined as sets of the form $\mathcal{P} = \{x: \; |\theta_{ij}| \leq \overline{\theta}_{ij}, \; \forall (i,j) \in \mathcal{E}\}$, with $\mathcal{P}_2 \subseteq \mathcal{P} \subseteq \mathcal{P}_1$, and $\mathcal{S} = \{x: \; |\omega_i|\leq \overline{\omega}_i, \; \forall i \in \mathcal{N}\}$. 

Now let $\eta_i$, for $i\in \mathcal{N}$, define a norm-bounded time-varying disturbance, i.e., $|\eta_i| \leq \overline{\eta}_i$, entering in the dynamics of the $i^{\text{th}}$ generator as follows: 
\begin{align}
\label{eq:modeldisturbance}
\begin{split}
&\dot{\theta_i} = \omega_i \\
&m_i \dot{\omega_i} + d_i \omega_i = P_{i}- \sum\limits_{j \in \mathcal{N}_i} y_{ij} \sin \theta_{ij} + \eta_i
\end{split}
\end{align}

Therefore, the disturbed system model can be written in a Lur'e form as $\dot{x}=A x - B \phi(z) + H \eta$, $z = C x$, where
\begin{equation}
    H = \begin{bmatrix}
    \mathbb{0} \\
    M^{-1}
    \end{bmatrix}
\end{equation}

The problem can be now mathematically formulated as follows. Let $x^*$, $x_f$ and $x(t)$ define the post-fault equilibrium point, the fault-cleared state (the state of the system when the fault is cleared), and the state of the system during the post-fault dynamics, respectively. Moreover, let $\mathcal{X}$ define a region in the state space, inside the operational constraints, i.e., $\mathcal{X} \subseteq \mathcal{P} \cap \mathcal{S}$. The analysis carried out in this paper concentrates on computing in a scalable way the maximum region $\mathcal{X}$ and the associated bound $\overline{\eta}$ on the disturbance such that the following two conditions hold:

\smallskip
\begin{itemize}
    \item[(i)] $x_f \in \mathcal{X}$\;\,and\,\;$\eta = \mathbb{0}$  \,\,\quad\quad\; $\Rightarrow$ \;$x(t) \in \mathcal{X}$\;\,and\,\;$x(t) \rightarrow x^*$
    \item[(ii)] $x_f \in \mathcal{X}$\;\,and\,\;$\lVert \eta \rVert_{\mathcal{L}_{\infty}} \leq \overline{\eta}$  \; $\Rightarrow$ \,\,$x(t) \in \mathcal{X}$
\end{itemize}
\smallskip

Notice that the first condition corresponds to the classical transient stability assessment using Direct Methods, with the difference that in our formulation the state is not allowed to violate the operational constraints. The second condition aims at computing a robustness measure of the transient stability, defined as the amount of disturbance $\eta$ that can enter in the post-fault system dynamics such that the trajectory $x$ does not leave the region $\mathcal{X}$.

\section{Transient stability analysis}
\label{sec:reachability}

In this section we concentrate on the first condition stated in Section \ref{probform}. We thus analyze the undisturbed system ($\eta = \mathbb{0}$) and propose a scalable method to construct an invariant set $\mathcal{X}$ inside the operational constraints. Moreover, we prove that in the absence of the disturbance, the post-fault dynamics converges to the post-fault equilibrium ($x(t) \rightarrow x^*$) if the fault-cleared state resides inside this set ($x_f \in \mathcal{X}$). In Section \ref{sec:robust} we will focus on the second condition and propose a robust stability certificate, by finding a bound $\overline{\eta}$ on the disturbance $\eta$ such that the set computed in this section remains invariant.

\subsection{Convex Lyapunov function}
\label{subsec:lyapunov}

In order to construct an invariant set $\mathcal{X}$ inside the operational constraints we use Lyapunov arguments. We propose the following Lyapunov function candidate:
\begin{equation}
\label{eq:lyapunov}
V(x) = x^T P x + 2 \sum\limits_{k=1}^{|\mathcal{E}|} \lambda_k \int_0^{z_k(x)} \left(\overline{\delta}_k s - \phi_k(s)\right) ds
\end{equation}
where $P = P^T \in \mathbb{R}^{n \times n}$ is positive definite, and $\lambda_k$, $k \in \{1,\ldots,|\mathcal{E}|\}$, are non-negative scalars.

In the following lemma we prove the convexity of $V(x)$, a feature of great importance for the scalability of the proposed method.

\begin{lemma}
\label{lemma:convexlyap}
The Lyapunov function candidate $V(x)$ defined in \eqref{eq:lyapunov}, with $\overline{\delta}_k = 1-\cos \theta_{ij}^*$, is strongly convex for all $x \in \mathbb{R}^{2n}$.
\end{lemma}
\begin{proof}
The Hession of $V(x)$ can be computed as $P+\diag\left(L,\mathbb{0}\right)$, with $L=E\diag(2\lambda_k(1-\cos\theta_{ij}))_{k\sim(i,j)\in\mathcal{E}}E^T$, where $k\sim(i,j)$ indicates that $k$ refers to the edge $(i,j)$. Since $1-\cos\theta_{ij} \geq 0, \; \forall (i,j)\in\mathcal{E}$, the matrix $L$ is a symmetric Laplacian matrix and therefore positive semi-definite. Therefore, the Hessian is lower-bounded by $\sigma_{min}(P) I$ and $V(x)$ is a strongly convex function.
\end{proof}

\subsection{Transient stability under operational constraints}
\label{subsec:reachconstr}

 Recall the sets $\mathcal{P}$ and $\mathcal{S}$ which correspond to the operational constraints. Let $\mathcal{P}_{ij}^-$ define the boundary of $\mathcal{P}$ corresponding to the equality $\theta_{ij} = -\overline{\theta}_{ij}$, and $\mathcal{P}_{ij}^+$ the boundary corresponding to the equality $\theta_{ij} = \overline{\theta}_{ij}$. We now consider the following set of $2 |\mathcal{E}|$ optimization problems, two for every edge $(i,j) \in \mathcal{E}$:
\begin{equation}
\label{eq:opt1}
\begin{aligned}
V_{ij}^- = \; & \underset{x}{\text{min}} 
& & V(x) \\
&\text{s.t.} 
& & x \in \mathcal{P}_{ij}^- \\
& & & \omega_i \leq \omega_j \\
\end{aligned} \quad\quad\quad
\begin{aligned}
V_{ij}^+ =\; & \underset{x}{\text{min}}
& & V(x) \\
&\text{s.t.} 
& & x \in \mathcal{P}_{ij}^+ \\
& & & \omega_i \geq \omega_j \\
\end{aligned}
\end{equation}

Now let $\mathcal{S}_{i}^-$ define the boundary of $\mathcal{S}$ corresponding to the equality $\omega_i = -\overline{\omega}_i$ and $\mathcal{S}_{i}^+$ the boundary corresponding to the equality $\omega_i = \overline{\omega}_i$. We now consider the following set of $2 n$ optimization problems, two for every node $i \in \mathcal{N}$:
\begin{equation}
\label{eq:opt2}
\begin{aligned}
W_{i}^- = \; & \underset{x}{\text{min}} 
& & V(x) \\
&\text{s.t.} 
& & x \in \mathcal{S}_{i}^- \\
\end{aligned} \quad\quad
\begin{aligned}
W_{i}^+ =\; & \underset{x}{\text{min}} 
& & V(x) \\
&\text{s.t.} 
& & x \in \mathcal{S}_{i}^+ \\
\end{aligned}
\end{equation}

Since $V(x)$ is a convex function and the constraints are linear, the optimization problems in \eqref{eq:opt1} and \eqref{eq:opt2} are convex. As a consequence, their solutions can be obtained in polynomial time.

Before stating the main result of this section, we define the matrices $\Lambda = \diag(\lambda_k)_{k \in \{1,\ldots,|\mathcal{E}|\}}$, $\underline{\Delta} = \diag(\underline{\delta}_k)_{k \in \{1,\ldots,|\mathcal{E}|\}}$, $\overline{\Delta} = \diag(\overline{\delta}_k)_{k \in \{1,\ldots,|\mathcal{E}|\}}$, and 
\begin{equation}
R = 
\begin{bmatrix}
R_{11} & R_{12} \\
R_{12}^T & R_{22}
\end{bmatrix}
\label{eq:R}
\end{equation}
where $R_{11}$, $R_{12}$ and $R_{22}$ are defined as:

$R_{11} \; := \; A^T \left(P+C^T \Lambda \overline{\Delta} C\right) + \left(P+C^T \overline{\Delta}  \Lambda C\right) A $

\quad\quad\quad\quad$ - 2 C^T \underline{\Delta} \Gamma \overline{\Delta} C $

$R_{12} \;:=\; -P B - A^T C^T \Lambda + C^T (\underline{\Delta}+\overline{\Delta}) \Gamma $

$R_{22} \;:=\; -2 \Gamma$

\noindent with $\Gamma \in \mathbb{R}^{|\mathcal{E}| \times |\mathcal{E}|}$ a positive definite diagonal matrix.

\begin{theorem}
\label{thm1}
Let $V_{\text{max}} = \text{min}\{V^*,W^*\}$, with $V^* = \text{min}\{V_{ij}^-,V_{ij}^+\}_{(i,j) \in \mathcal{E}}$ and $W^* = \text{min}\{W_{i}^-,W_{i}^+\}_{i \in \mathcal{N}}$. If $R \preceq 0$, the set $\mathcal{X} = \{x: \; V(x) \leq V_{\text{max}}\} \, \cap \, \mathcal{P}$ is an invariant set inside the operational constraints. Moreover, any trajectory of the system \eqref{eq:model} originating inside this set converges to the post-fault equilibrium point.
\end{theorem}
\begin{proof}
The derivative of $V$ along the system trajectory, $\dot{V}(x)= \dot{x}^T P x + x^T P \dot{x} + (\overline{\Delta} z - \phi)^T\Lambda \dot{z} + \dot{z}^T \Lambda (\overline{\Delta} z - \phi)$, can be written as $\dot{V}(x)=[x^T \; \phi^T] \; Q \; [x^T \; \phi^T]^T$, with
\begin{equation*}
Q = 
\begin{bmatrix}
Q_{11} & Q_{12} \\
Q_{12}^T & \mathbb{0}
\end{bmatrix}
\end{equation*}
where

$Q_{11} \; := \; A^T \left(P+C^T \Lambda \overline{\Delta} C\right) + \left(P+C^T \overline{\Delta}  \Lambda C\right) A$

$Q_{12} \; := \; -P B - A^T C^T \Lambda$

The sector bound \eqref{eq:sector} implies $2 (\phi-\overline{\Delta} z)^T \Gamma (\phi-\underline{\Delta} z) \leq 0$ inside $\mathcal{P}$, for any diagonal matrix $\Gamma \succeq 0$. This condition can be written in an LMI form as $[x^T \; \phi^T] \; \tilde{Q} \; [x^T \; \phi^T]^T \leq 0$, with
\begin{equation*}
\tilde{Q} = 
\begin{bmatrix}
2 C^T \underline{\Delta} \Gamma \overline{\Delta} C & - C^T (\underline{\Delta}+\overline{\Delta}) \Gamma \\
-\Gamma (\underline{\Delta}+\overline{\Delta}) C& 2 \Gamma
\end{bmatrix}
\end{equation*}

Notice that $[x^T \; \phi^T] \; \tilde{Q} \; [x^T \; \phi^T]^T = 0$ only at the equilibrium and at some points on the boundary of $\mathcal{P}$. Using the S-lemma, we obtain the sufficient stability condition $R = Q-\tilde{Q} \preceq 0$ inside $\mathcal{P}$. Therefore, if $R \preceq 0$, then $\dot{V}(x) < 0$, and the Lyapunov function $V(x)$ is decreasing inside $\mathcal{P}$.

Now, notice that $V(x)=W^*$ is the minimum level set of $V(x)$ that intersects the boundary of $\mathcal{S}$. On the other hand, $V(x)=V^*$ is the minimum level set of $V(x)$ that intersects the out-flow boundary of $\mathcal{P}$, composed by the boundary segments characterized by $|\theta_{ij}| = \overline{\theta}_{ij}$ and $\theta_{ij}\dot{\theta}_{ij}=\theta_{ij}(\omega_i-\omega_j) \geq 0$. By doing so, the level set $V(x)=V^*$ may intersect the boundary of $\mathcal{P}$, but only on the segments which don't allow the system trajectory to escape $\mathcal{P}$. 

In conclusion, $\mathcal{X} = \{x: \; V(x) \leq V_{\text{max}}\} \, \cap \, \mathcal{P}$ is an invariant set contained in $\mathcal{P} \; \cap \; \mathcal{S}$.
\end{proof}

Therefore, an invariant set contained inside the polytope of operational constraints is computed using the set of $2 (|\mathcal{E}|+n)$ convex optimizations \eqref{eq:opt1} and \eqref{eq:opt2}, whenever the LMI condition $R \preceq 0$ is satisfied.

\section{Robustness measure of transient stability}
\label{sec:robust}

In this section we concentrate on the second condition stated in Section \ref{probform}.
We thus analyze the effect of the disturbance $\eta$ on the dynamics of the power system. Specifically, we will find an $\mathcal{L}_{\infty}$ bound $\overline{\eta}$ on $\eta$ such that the set $\mathcal{X}$ defined in Theorem \ref{thm1} remains invariant. The methodology used builds upon a local version of the Input-to-State Stability (ISS) theory, as described in the following.

\subsection{Local Input-to-State Stability}
\label{subsec:liss}

Before moving on to the stability concepts, we recall the definitions of comparison functions. A continuous function $\gamma : \mathbb{R}_+ \rightarrow \mathbb{R}_+$ is said to be of class $\mathcal{K}$ if it is strictly increasing and $\gamma(0)=0$. It is of class $\mathcal{K}_{\infty}$ if, in addition, it is unbounded. A continuous function $\beta : \mathbb{R}_+ \times \mathbb{R}_+ \rightarrow \mathbb{R}_+$ is said to be of class $\mathcal{KL}$ if, for fixed $t$, the function $\beta(\cdot,t)$ is of class $\mathcal{K}$ and, for fixed $s$, the function $\beta(s,\cdot)$ is strictly decreasing and tends to $0$.

Recall the power system dynamics written in Lur'e form:
\begin{equation}
\label{eq:modelLISS}
\dot{x} = A x - B \phi(C x) + H \eta
\end{equation}

Let $\Omega$ define a local region of initial states for the system \eqref{eq:modelLISS} with $\eta = \mathbb{0}$, i.e., a compact positively invariant set (containing the post-fault equilibrium point as an interior point). Moreover, let $\Xi$ define a local region of external inputs (disturbances) $\eta$, defined as $\Xi = \{\eta: \lVert \eta \rVert_{\mathcal{L}_{\infty}} \leq \overline{\eta}\}$.

In the following we will introduce the concepts of local Input-to-State Stability (LISS) and LISS-Lyapunov function.

\begin{definition}
The system \eqref{eq:modelLISS} is LISS if there exist functions $\beta \in \mathcal{KL}$ and $\gamma \in \mathcal{K}$ such that for any initial state $x_0 \in \Omega$ and disturbance $\eta \in \Xi$,
\begin{equation}
\lVert x(t, x_0, \eta)\rVert \leq \beta(\lVert x_0\rVert, t)+\gamma(\lVert \eta \rVert_{\mathcal{L}_{\infty}}), \quad \forall \; t \geq 0
\end{equation}
\end{definition}

\begin{definition}
\label{def2}
A smooth function $V(x):\Omega \rightarrow \mathbb{R}_{+}$ is called a LISS-Lyapunov function if there exist functions $\psi_1$, $\psi_2$ $\in \mathcal{K}_{\infty}$, $\chi$ $\in \mathcal{K}_{\infty}$, and $\psi$ $\in \mathcal{K}$ such that for any initial state $x_0 \in \Omega$ and disturbance $\eta \in \Xi$,
\begin{align}
\label{eq:lissLyap}
\begin{split}
& \psi_1(\lVert x\rVert) \leq V(x) \leq \psi_2(\lVert x\rVert), \quad \forall \; x \in \Omega \\
& \dot{V}(x) \leq -\psi(\lVert x\rVert), \quad \forall \; x \in \Omega, \; \lVert x\rVert \geq \chi(\lVert \eta\rVert_{\mathcal{L}_{\infty}})
\end{split}
\end{align}
\end{definition}

These two concepts are equivalent, as stated in the following lemma.

\begin{lemma}[\cite{mironchenko2016local}]
\label{lemma:LISS}
 The system \eqref{eq:modelLISS} is LISS if and only if it admits a LISS-Lyapunov function.
\end{lemma}

These concepts will be used in the following to prove that the system \eqref{eq:modelLISS} is LISS, and to find a robust stability certificate for the transient stability assessment.

\subsection{Robust stability certificate}
\label{subsec:certificate}

In the following we use the concepts stated above to propose a method to compute a local region of external disturbances $\Xi$, such that for $\eta \in \Xi$ and $\Omega = \mathcal{X}$ (with $\mathcal{X}$ defined in Theorem \ref{thm1}), the system \eqref{eq:modelLISS} is LISS.

Consider the following set of $2|\mathcal{E}|$ convex optimization problems:
\begin{equation}
\label{eq:opt3}
\begin{aligned}
\hat{V}_{ij}^- = \; & \underset{x}{\text{min}} 
& & V(x) \\
&\text{s.t.} 
& & x \in \mathcal{P}_{ij}^- \\
\end{aligned} \quad\quad\quad
\begin{aligned}
\hat{V}_{ij}^+ =\; & \underset{x}{\text{min}} 
& & V(x) \\
&\text{s.t.} 
& & x \in \mathcal{P}_{ij}^+ \\
\end{aligned}
\end{equation}

Now let $\hat{V}^* = \text{min}\{\hat{V}_{ij}^-,\hat{V}_{ij}^+\}_{(i,j) \in \mathcal{E}}$. We define $\hat{V}_{\text{max}} = \text{min}\{\hat{V}^*,W^*\}$. Notice that $V(x) = \hat{V}_{\text{max}}$ defines the maximum level set of $V(x)$ which resides completely inside the polytope $\mathcal{P} \; \cap \; \mathcal{S}$. As a consequence, $\hat{V}_{\text{max}} \leq V_{\text{max}}$.

We are now ready to state the main result of this section.

\begin{theorem}
\label{thm2}
Consider the local region of initial states $\Omega = \mathcal{X}$, with $\mathcal{X}$ defined in Theorem \ref{thm1}. If the matrix $R$ in \eqref{eq:R} satisfies $R \prec 0$, then the system \eqref{eq:modelLISS} is LISS with a local region of external disturbances $\Xi = \{\eta: \lVert \eta \rVert_{\mathcal{L}_{\infty}} < \overline{\eta}\}$,
\begin{equation}
\label{eq:eta}
\overline{\eta} = \frac{\sigma_{min}(-R)}{2\lVert P H \rVert\sqrt{\sigma_{max}(P) + \mu \, \lVert C \rVert^2}}\sqrt{\hat{V}_{\text{max}}}
\end{equation}
where $\mu = \text{max}\{\lambda_k (\overline{\delta}_k-\underline{\delta}_k)\}_{k \in \{\,\ldots |\mathcal{E}|\}}$.
\end{theorem}
\begin{proof}
The function $V(x)$ can be lower and upper bounded $\psi_1 \lVert x \rVert^2 \leq V(x) \leq \psi_2 \lVert x \rVert^2$, with $\psi_1 = \sigma_{\text{min}}(P)$ and $\psi_2 = \sigma_{\text{max}}(P)+\mu \lVert C \rVert^2$.

The lower bound is obtained from the quadratic function $x^T P x$, since $\lambda_k \left(\overline{\delta}_k s - \phi_k(s)\right) \geq 0$. The upper bound is obtained in a similar way, by noticing that $\lambda_k \left(\overline{\delta}_k s - \phi_k(s)\right) \leq \lambda_k \left(\overline{\delta}_k s - \underline{\delta}_k s\right) \leq \mu s$.

The derivative of $V(x)$ with respect to time is
\begin{equation}
\dot{V}(x) = 
\begin{bmatrix}
x \\
\phi
\end{bmatrix}^T
Q 
\begin{bmatrix}
x \\
\phi
\end{bmatrix}
+ 2 x^T P H \eta
\end{equation}
with matrix $Q$ defined in the proof of Theorem \ref{thm1}. Therefore $\dot{V}(x)$ can be locally bounded as
\begin{equation}
\dot{V}(x) \leq -\sigma_{\text{min}}(-R) \lVert x \rVert^2 + 2 \lVert P H \rVert \lVert x \rVert \lVert \eta \rVert_{\mathcal{L}_{\infty}}
\end{equation}

Now, for $\eta \in \Xi = \{\eta: \lVert \eta \rVert_{\mathcal{L}_{\infty}} < \overline{\eta}\}$, with $\overline{\eta}$ defined in \eqref{eq:eta}, $V(x)$ can be rewritten in the form of Definition \ref{def2}, concluding that $V(x)$ is a LISS-Lyapunov function and the system \eqref{eq:modelLISS} is LISS. 

Notice that the level set used in \eqref{eq:eta} to compute the upper bound on $\overline{\eta}$ is $V(x) = \hat{V}_{\text{max}}$. The reason why we cannot use $V_{\text{max}}$ is because $V(x) = V_{\text{max}}$ intersects the boundary of $\mathcal{P}$ and, with the disturbance $\eta$, the trajectory may not be pushed back inside of $\mathcal{P}$ once it touches its boundary. Therefore, a level set which is completely inside the operational constraints needs to be used.
\end{proof}

This result ensures that the set $\mathcal{X}$ remains invariant for any disturbance $\eta$, with $\lVert \eta \rVert_{\mathcal{L}_{\infty}} \leq \overline{\eta}$. The following corollary translates Theorem \ref{thm1} and Theorem \ref{thm2} in a robust transient stability certificate for power systems.

\begin{corollary}
Consider a power system described by the Lur'e representation model \eqref{eq:modelLISS}. Let $|\theta_{ij}| \leq \overline{\theta}_{ij}$, $\forall (i,j) \in \mathcal{E}$, and $|\omega_i| \leq \overline{\omega}_i$, $\forall i \in \mathcal{N}$, define some operational constraints on the angle differences of neighboring generators and frequencies, respectively. Moreover, let $\eta$ define a time-varying disturbance, entering in the post-fault system dynamics as described in Section \ref{probform}. If $R \prec 0$, and the fault-cleared state resides in the set $\mathcal{X}$, and $\lVert \eta \rVert_{\mathcal{L}_{\infty}} < \overline{\eta}$, with $\mathcal{X}$ and $\overline{\eta}$ defined in Theorem \ref{thm1} and \ref{thm2}, then the post-fault trajectory will never escape the set $\mathcal{X}$, i.e., will never violate the operational constraints. Moreover, the post-fault trajectory will converge to the set $\{x: \; V(x) < \hat{V}_{\text{max}}\}$. In particular, if $\eta = \mathbb{0}$, the post-fault trajectory will converge to the post-fault equilibrium point.
\end{corollary}

\section{Numerical Validation}
\label{sec:numerical}

This section presents a numerical validation of the methodology proposed in this paper. For illustration, we consider a single-machine infinite-bus (SMIB) system and plot its phase portrait in Figure \ref{fig:reachability}. The dashed black contour represents the biggest invariant set, computed as the maximum level set of the Lyapunov function $V(x)$ inside the region of attraction. Notice that although this level set covers a big portion of the region of attraction, it violates the operational constraints $\lvert\theta\rvert \leq 3\pi/4$ rad and $\lvert\omega\rvert \leq \pi$ rad/s, represented by the dashed horizontal and vertical lines.

\begin{figure}[t]
	\centering
\includegraphics[width=0.75\columnwidth]{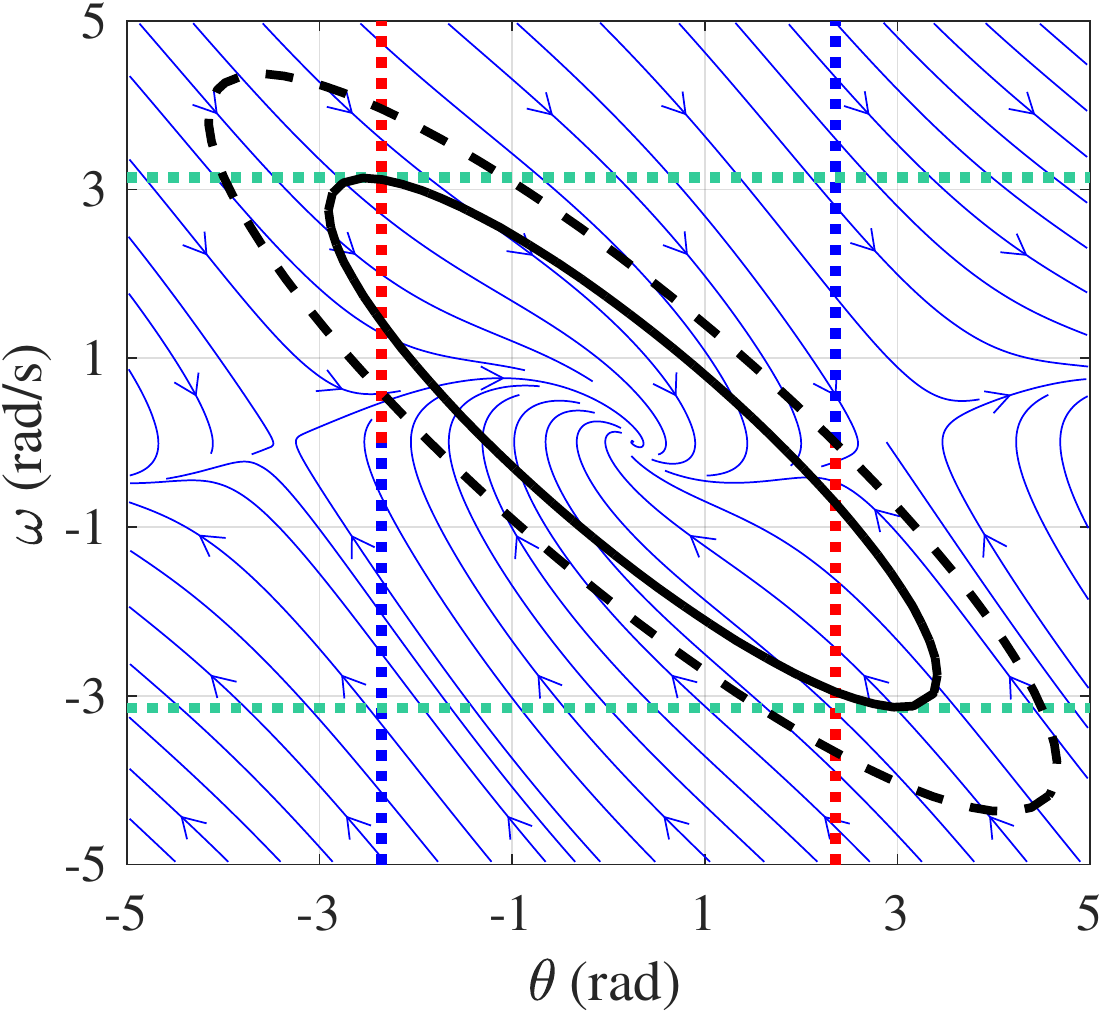}
\caption{Phase portrait and invariant sets for a SMIB system.}
\label{fig:reachability}
\end{figure}

Another level set is then computed using the methods presented in Section \ref{sec:reachability}, and plotted with a continuous black color. Notice that this new level set intersects the angle constraints only on the red dashed half-lines, which forbid the system trajectories to escape. The invariant set $\mathcal{X}$ is defined as the intersection of the new sublevel set with the region confined by the angle constraints. It can be seen from the phase portrait that any system trajectory starting inside $\mathcal{X}$ will converge to the post-fault equilibrium point.

We now consider the IEEE 9 bus system and compute a robustness measure of transient stability, i.e., an $\mathcal{L}_{\infty}$-norm bound $\overline{\eta}$ on the disturbance $\eta$ entering in the post-fault system dynamics such that the set $\mathcal{X}$ remains invariant. A Kron-reduction was applied to the network, and damping coefficients of $10$ p.u. were added to reflect the effect of the primary frequency control action. We consider the following operational constraints: $\lvert\theta_{ij}\rvert \leq \pi/6$ rad, $\forall (i,j) \in \mathcal{E}$, and $\lvert\omega_i\rvert \leq \pi$ rad/s, $\forall i \in \mathcal{N}$. Using the expression \eqref{eq:eta}, presented in Theorem \ref{thm2}, we obtain $\overline{\eta} = 0.0026$.

Notice that such a small value is expected and is mainly due to the fact that $\overline{\eta}$ represents the bound on the maximum disturbance magnitude entering in the post-fault trajectory originating at the worst-case fault-cleared state inside the set $\mathcal{X}$. Since the fault-cleared state can reside very close to the actual boundary of the region of attraction, a very small disturbance can push the trajectory outside of the region of attraction, allowing therefore a very small robustness margin.

In Figure \ref{fig:tds}, the time domain simulation confirms that if the fault-cleared state resides inside the set $\mathcal{X}$, the frequency constraints are not violated.

\begin{figure}[t]
	\centering
\includegraphics[width=0.85\columnwidth]{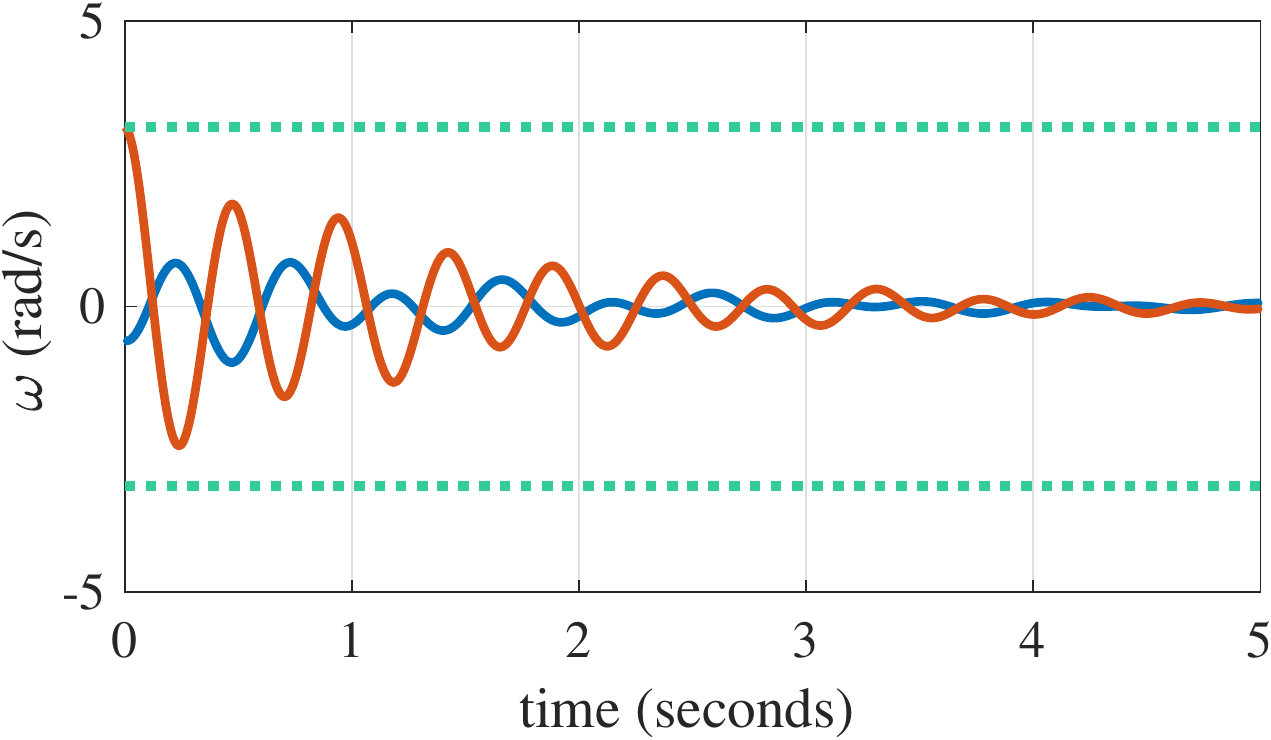}
\caption{Generators frequency dynamics for the IEEE 9 bus system with fault-cleared state inside the set $\mathcal{X}$.}
\label{fig:tds}
\end{figure}

\section{Conclusion}

In this paper we considered the problem of quantifying the robustness of power systems transient stability, with consideration of operational constraints over frequency deviation and angular separation. To this end, we first proposed a novel convex Lyapunov function, which we employed to efficiently compute an invariant set inside the region of attraction of the post-fault equilibrium point and the operational constraints. If the fault-cleared state resides inside this set, the undisturbed post-fault dynamics will converge to the post-fault equilibrium point, and it will never violate the operational constraints.

We then used local ISS notions and proposed a bound on the magnitude of the disturbance entering in the post-fault system dynamics such that the previously computed set remains invariant. If the fault-cleared state resides inside this set, the disturbed post-fault dynamics will never escape, i.e., it will never leave the region of attraction of the post-fault equilibrium point or violate the operational constraints.

\bibliographystyle{IEEEtran}
\bibliography{lff} 

\end{document}